\documentclass[10pt]{article}

\usepackage{amssymb}
\usepackage{amsfonts}
\usepackage{amsmath}
\usepackage{amsthm}
\usepackage[english]{babel}
\usepackage[latin1]{inputenc}
\usepackage[T1]{fontenc}
\usepackage{graphicx}
\usepackage{verbatim}
\usepackage{setspace}
\usepackage{epsfig}
\usepackage{kantlipsum}

\usepackage{color} 

\newtheorem{teo}{Theorem}[section]

\theoremstyle{definition}
\newtheorem{defi}[teo]{Definition}

\newcommand{\N}{\mathbb{N}}

\newcommand{\R}{\mathbb{R}}

\newcommand{\Lap}{\Delta}

\newcommand{\fr}{\partial}
\newcommand{\grad}{\nabla}
\newcommand{\dx}{\,dx}

\newcommand{\dtau}{\,d \tau}

\newcommand{\Div}{\operatorname{div}}

\newcommand{\nl}{\mbox{}\\}

\allowdisplaybreaks

\begin{document}

	\onehalfspace
	\mbox{} \vspace{-2.750cm} \\
	\mbox{} \vspace{-2.000cm} \\
	\begin{center}
		{\bf \Large Decay estimates for solutions of porous medium equations with advection}\\
		%
		\mbox{} \vspace{-0.50cm} \\
		\nl
		{\sc N. M. L. Diehl$\mbox{}^1$,
			L. Fabris$\mbox{}^2$ and
			J. S. Ziebell$\mbox{}^3$}\\
		
		$\mbox{}^{1}${\small
			Instituto Federal de Educa\c c\~ao, Ci\^encia e Tecnologia
			do Rio Grande do Sul} \\
		\mbox{} \vspace{-0.685cm} \\
		{\small Canoas, RS 92412, Brazil} \\
		\mbox{} \vspace{-0.390cm} \\
		$\mbox{}^{2}${\small
			Departamento de Matemática} \\
		\mbox{} \vspace{-0.685cm} \\
		{\small
			Universidade Federal de Santa Maria - Cachoeira do Sul} \\
		\mbox{} \vspace{-0.685cm} \\
		{\small
			Cachoeira do Sul, RS 96501, Brazil} \\
		\mbox{} \vspace{-0.390cm} \\
		$\mbox{}^{3}${\small
			Departamento de Matem\'atica Pura e Aplicada} \\
		\mbox{} \vspace{-0.685cm} \\
		{\small
			Universidade Federal do Rio Grande do Sul} \\
		\mbox{} \vspace{-0.685cm} \\
		{\small
			Porto Alegre, RS 91509, Brazil} \\
		\nl
		\mbox{} \vspace{-0.50cm} \\
		
		\noindent	This is a pre-print of an article published in Acta Applicandae Mathematicae. The final authenticated version is available online at: https://doi.org/10.1007/s10440-019-00246-4 \\\vspace{0.5cm}
		
		{\bf Abstract} \\
		\mbox{} \vspace{-0.200cm} \\
		\begin{minipage}{12.250cm}
			{\footnotesize
				\mbox{} \hspace{0.250cm}
				In this paper, we show that bounded weak solutions of the Cauchy problem for general degenerate parabolic equations
				of the form \\
				\mbox{} \vspace{-0.625cm} \\
				\begin{equation}
				\notag
				u_t \,+\;
				\mbox{div}\,f(x,t,u)
				\;=\;
				\mbox{div}\,(\;\!|\,u\,|^{\alpha} \, \nabla u \;\!),
				\quad \;\;
				x \in \mathbb{R}^{n}\!\:\!, \; t > 0,
				\end{equation}
				\mbox{} \vspace{-0.575cm} \\
				where
				$ \alpha > 0 $ is constant, decrease to zero,
				under fairly broad conditions on the advection flux $f$. Besides that, we derive a time decay rate for these solutions. \\
			}
		\end{minipage}
	\end{center}
	

	\section{ Introduction} 
	
	The main goal of this work is to obtain an optimal rate decay for (signed) weak solutions of the problem
	\begin{align}\label{i001}  \notag 
	u_t\;\! +\,
	\Div  f(x,t,u) 
	\,&=\;
	\Div \bigl(\;\!|\;\! u\,|^{\alpha} \, \nabla u\;\!\bigr) \quad \;\,
	x \in \mathbb{R}^{n} \!\;\!, \; t > 0,\\
	u(\cdot,0) \,&=\,
	u_{0} \in
	L^{p_{\mbox{}_{0}}}\!\;\!(\mathbb{R}^{n})
	\;\!\cap\;\!
	L^{\infty}(\mathbb{R}^{n}),
	\end{align}
	\mbox{} \vspace{-0.200cm} \\
	given constants
	$ \;\! \alpha > 0 $ and $  1 \leq p_{\mbox{}_{0}} \!< \infty \, $,
	and $ f \in C^1 $ satisfying
	
	\begin{equation}\label{i003}
	\sum_{j\,=\,1}^{n} \;\!
	\frac{\partial \;\!f_{\scriptstyle j}}
	{\partial \:\!x_{\scriptstyle j}}
	(x,t,\mbox{u}) \, \mbox{u}
	\,\geq\, 0
	\qquad
	\forall \;\,
	x \in \mathbb{R}^{n} \!\;\!, \; t \geq 0, \;
	\mbox{u} \in \mathbb{R},
	\end{equation}
	
	In the case of $f$ not depending on $x$ and $t$,
	we have, in particular, the problem
	\begin{align}\label{i014}
	u_t\;\! +\,
	\Div \, f(u)
	\:&=\;
	\Div \bigl(\;\!|\;\! u\,|^{\alpha} \, \nabla u\;\!\bigr)
	\quad \;\,
	x \in \mathbb{R}^{n} \!\;\!, \; t > 0, \notag \\
	u(\cdot,0) \,&=\,
	u_{0} \in
	L^{p_{\mbox{}_{0}}}\!\;\!(\mathbb{R}^{n})
	\;\!\cap\;\!
	L^{\infty}(\mathbb{R}^{n}),
	\end{align}
	whose solutions exhibit a lot of known properties of parabolic problems in a conservative way as, for example, regularity, decay in $\mbox{\small $L$}^{\!1}\!\:\!$ norm, mass conservation and comparison properties. If condition \eqref{i003} does not hold, some of these properties are no longer valid in general, such as decay in  $\mbox{\small $L$}^q\!\;\!$ norm for $ q > 1 $, 
	contrativity in $\mbox{\small $L$}^{1}\!\;\!$, global existence, decay to zero in various norms when $t\rightarrow\infty$, in case of global existence, etc.

	In fact, problem \eqref{i001} could be much more complicated when the condition \eqref{i003} is violated,
	as we indicate below. As an example, we consider, for simplicity, the one-dimensional problem
	\begin{align}\label{i005}
	u_t\;\! +\, 
	(\;\! f(x)\;\!|\,u\,|^{k} \:\!u \;\!)_{x}
	\;\!&=\:
	(\;\! |\;\!u\;\!|^{\:\!\alpha} \;\!u_x \:\!)_{x} \quad \;\,
	x \in \mathbb{R} \!\;\!, \; t > 0, \notag  \\ 
	u(\cdot,0) &= u_{0} \in L^{1}(\mathbb{R}) \cap L^{\infty}(\mathbb{R}) ,
	\end{align}
	
	\noindent where
	$ J \!:= \{\;\! x \in \mathbb{R} \!\!\;\!:\;\! f^{\prime}(x) < 0 \;\!\} 
	\neq \emptyset $.
	
	Rewriting the first equation as 
	\begin{equation}\label{i006} 
	u_t \;\!+\, (k + 1) \, f(x) \, |\;\!u\;\!|^{k} u_x
	\;\!=\:
	(\;\! |\;\!u \;\!|^{\:\!\alpha} \;\!u_x \:\!)_{x}
	\:\!-\, f^{\prime}(x) \;\! |\;\!u\;\!|^{\:\!k} u,
	\end{equation}
	we can see that $ u(x,t) $ tends to be stimulated to grow in magnitude at the points where $ x \in J $, in particular where $ - \,f^{\prime}(x) \gg 1 $. On the other hand, as\linebreak
	$ {\displaystyle
		\|\, u(\cdot,t) \,
		\|_{\mbox{}_{\scriptstyle L^{1}(\mathbb{R})}}
		\!\;\!\leq\:\!
		\|\, u_0 \;\!
		\|_{\mbox{}_{\scriptstyle L^{1}(\mathbb{R})}}
	} $
	for all $t$ (while the solution exists, see \cite{Fabris2013}), an intense growth in a given part of $J$ results in the formation of elongated structures (as is illustrated in fig.$\;$1), that tend to be efficiently dissipated by the diffusivity term present.
	%
	%
	%
	\begin{figure}[!ht]
		\centering   
		\includegraphics[width = 8.00cm]{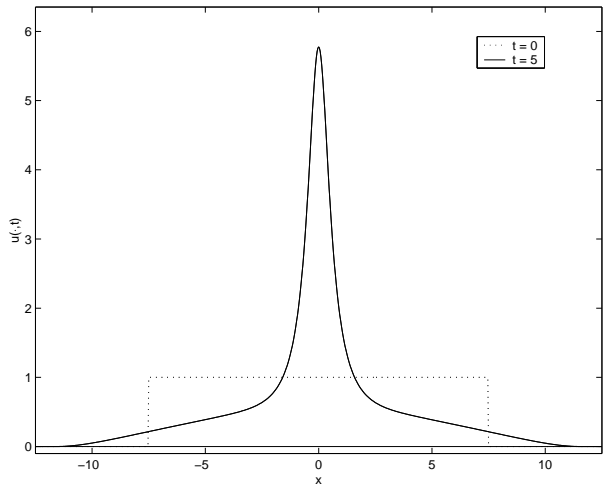}
		\caption{ \footnotesize
			Representation of the solution $ u(\cdot,t) $
			on the instant $ t = 5 $ (full curve)
			corresponding to equation \eqref{i006} above with
			$ f(x) = -\, \mbox{tgh}\,x $, $ \alpha = 0.5$,
			$ k = 1.5 $, and initial state $ u_0 $ indicated (dashed curve).
			We note the growth of $ u(\cdot,t) $
			due $ f^{\prime}(x) < 0 $, with formation of elongated structures
			("High frequency waves
			")
			due the mass conservation. }
	\end{figure}
	%
	%
	The greater the growth of $ |\,u(x,t) \,| $, the greater will be the effect of the term
	$ -\, f^{\prime}(x) \;\! |\;\!u\;\!|^{\:\!k} u $
	on the right side of \eqref{i006} in forcing the additional growth and greater will be the dissipative capacity of the diffusivity term in \eqref{i006} to inhibit such growth, given the increase of the own diffusion coefficient and of the intensification of the stretching effects on the profile of $ u(\cdot,t)$!

	The competition between the diffusive and the forcing terms in the equation \eqref{i005} can, in this way, become so intense that the final result of this interaction (explosion or not on finite time, global existence and the behavior when  $ t \rightarrow \infty $, etc) is very difficult to be foreseen. 

	Besides that, in contrast to the current literature (see e.g.$\;$\cite{Hu2011,QuittnerSouplet2007,SamarskiiGalaktionov1995}), this kind of interaction (with mass conservation or similar links) only began to be investigated mathematically very recently (in \cite{BarrionuevoOliveiraZingano2014,BrazMeloZingano2015,Fabris2013}).

	In the case of  globally defined solutions, we can examine other open questions also for the problem \eqref{i001} with the condition \eqref{i003} even in $n=1$ dimension. For example, we don't know general conditions about $ \,\!f, \,\!u_0 \,\!$ that prevent blow-up at infinity, $[$ that is,
	in order to have 
	$ u(\cdot,t) \in L^{\infty}([\;\!0, \infty), L^{\infty}(\mathbb{R}^{n})) \;\!]$ ,
	or conditions that ensure asymptotic decay 
	$ [ {\displaystyle
		\;\! \lim_{t\,\rightarrow\,\infty}
		\|\, u(\cdot,t) \,
		\|_{\mbox{}_{\scriptstyle L^{\infty}(\mathbb{R}^{n})}}
		\!=\;\! 0 \,
	} ] $, or convergence to stationary states (when they exist), and so on.
	
	These questions will not be examined in this article, with one exception: we will show that the condition \eqref{i003} above ensures the {\em decay}:
	given
	$ u_0 \in L^{p_{\mbox{}_{0}}}(\mathbb{R}^{n}) $,
	the solution $ u(\cdot,t) \in C^{0}([\;\!0,\infty), L^{p_{\mbox{}_{0}}}(\mathbb{R}^{n})) $
	corresponding to the problem \eqref{i001} with the condition \eqref{i003}
	satisfies the smoothing effect \\
	\mbox{} \vspace{-0.600cm} \\
	\begin{equation}\label{i011}
	\|\, u(\cdot,t) \,\|_{\mbox{}_{\scriptstyle L^{\infty}(\mathbb{R}^{n})}}
	\:\!\leq\:
	K\!\:\!(n,p_{\mbox{}_{0}} \!\:\!,\alpha) \,
	\|\, u(\cdot,0) \,
	\|_{\mbox{}_{\scriptstyle L^{p_{\mbox{}_{0}} \!\:\!}(\mathbb{R}^{n})}}
	^{\mbox{}^{\scriptstyle \;\!\delta_{\mbox{}_{0}}}}
	\;\!
	t^{\mbox{}^{\scriptstyle -\, \gamma_{\mbox{}_{0}}}}
	\quad
	\forall \;\, t > 0,
	\end{equation}
	\mbox{} \vspace{-0.400cm} \\
	where \\
	\mbox{} \vspace{-1.100cm} \\
	\begin{equation}\label{i012}
	\delta_{\mbox{}_{0}} \;\!=\;
	\mbox{\small $ {\displaystyle
			\frac{2 \;\!p_{\mbox{}_{0}} }
			{\;\!2 \;\!p_{\mbox{}_{0}} \,\! + \;\! n \alpha}
			\;\!} $},
	\qquad
	\gamma_{\mbox{}_{0}} \;\!=\;
	\mbox{\small $ {\displaystyle
			\frac{n}
			{\;\!2 \;\!p_{\mbox{}_{0}} \,\! +\;\! n \alpha}
			\;\!} $},
	\end{equation}
	\mbox{} \vspace{-0.050cm} \\
	and where
	$ K\!\;\!(n,p_{\mbox{}_{0}} \!\:\!,\alpha) > 0 $
	is a constant that depends
	only on the parameters $ n, p_{\mbox{}_{0}} \!\:\!, \alpha $ 
	(and not on $\;\!t, \;\!u, \;\! u_0, \! $  \text{ or } $f$). 
	In fact, this decay rate is optimal. Since Barenblatt's solutions decay at this rate, see e.g. \cite{Vazquez2006}, and it is solutions of the Cauchy problem whit $f \equiv 0 $ . 
	
	This kind of estimate has already been proved for similar problems, without advection. It is well known that the weak solutions of the problem
	\begin{align} \notag
	u_t&= \Div (|u|^{m-1}\nabla u)\,\,\mbox{in}\,\, \mathbb{R}^n\times (0,+\infty)\\ 
	u(x,0)&=u_0(x)\,\,\mbox{for}\,\,x\in\mathbb{R}^n.
	\label{pmevazquez}
	\end{align}
	defined on a certain domain $M\subseteq\mathbb{R}^n$, have some smoothing effects as
	\begin{align*}
	\|u(t)\|_p\leq C\|u_0\|_q^{\gamma}\,\,t^{-\beta}
	\end{align*}	
	where $p>q\geq 1$, $C$ is a constant, $\gamma$ and $\beta$ are appropriate positive expoents and for $m>0$, $m(p-1)>1$ \cite{bonforte2006}, or for $m>1$  \cite{bonforte2005,Vazquez2006}.

	This kind of bound were also showed, for $m>1$,  in \cite{fotache2017,grillo2012,grilloporzio2012}   when we add the Newmann condition
	$$
	\frac{\partial u^m}{\partial n}=0 \,\,\,\mbox{in}\,\, \partial M\times (0,+\infty)
	$$
	or, in \cite{porzio,grilloporzio2012}, with  the Dirichlet condition  
	$$
	u=0 \,\,\,\mbox{in}\,\, \partial M\times (0,+\infty),
	$$
	
	\noindent in cases where $x \in M \neq \R^n $.  \\
	
	Now, returning to our problem, let us consider the problem with advection 
	\begin{equation}\label{limitacaolq}
	\begin{aligned}
	u_t + \Div f(x,t,u) &= \Div (|\, u \,|^{\alpha}  \grad u ) ,
	\quad  x \in \R^{n}, \;  t > 0, \quad  \\
	u(x,0) &= u_0 \in L^{p_0}(\R^n) \cap  L^{\infty}(\R^n),
	\end{aligned}
	\end{equation}
	where $1 \leq p_0$ and $\alpha >0$.
	This paper is organized as follows. Section 2 is devoted to showing an estimate for $L^q$ norms of the solutions of this problem. An energy inequality and some decay estimates  for $L^p$ and $L^q$ norms for each $p_0\leq q \leq\infty$ and $0\leq t\leq T$ are presented in Section 3. In both Sections, we will consider $u_0>0$ (or $u_0<0$) for all $x \in \R^n$.  Finally, Section 4 is devoted to find an optimal decay rate for weak solutions with any $u_0\in L^{p_0}(\mathbb{R})$.   
	
	We remark that, in this paper, we understand as smooth and weak solution to the problem \eqref{limitacaolq} a function that satisfies the following definitions, respectively:
	
	\begin{defi}\label{106} A smooth function $u \in L^{\infty}_{loc}\big( [0,T_*),L^{\infty}(\R^n) \big) $ is a bounded classical solution in a maximal interval of existence $[0, T_*)$, where \linebreak $0 \leq T_* \leq \infty $, if it satisfies classically the first equation of \eqref{limitacaolq} and, besides that, \linebreak $u(\cdot, t) \rightarrow u_0 $ in $L^{p_0}_{loc}(\R^n)$, when $t\rightarrow 0 $.  
	\end{defi}
	
	\begin{defi} A weak solution to the problem \eqref{limitacaolq} is a function $u$ 
		that satisfies
		$$ \int^{T}_{0} \int_{\R^n} \!\! u(x,\tau) \Psi_t(x,\tau) +   \langle f(x,\tau,u) , \grad \Psi(x,\tau) \rangle + \frac{|u(x,\tau)|}{\alpha +1}^{\alpha +1 } \!\! \Lap \Psi(x,\tau)  \dx  \dtau  = 0 ,  $$
		\noindent for any $\Psi \in C_0^{\infty}(\R^n \times (0,T) )$ and $u(\cdot ,t) \rightarrow u_0 $ in $L_{loc}^{p_0}$, when $t \rightarrow 0$.
	\end{defi}
	
	When $u_0 >0 $ (or $u_0 < 0$) for all $x \in \R^n$, the solutions of the problem \eqref{limitacaolq} are strictly positive (or strictly negative), see \cite{Fabris2013}. In these cases the solutions are smooth, which is the reason why we consider smooth solutions in Sections 2 and 3, and weak solutions in Section 4 where the initial condition and solutions can change sign. For a more complete discussion of regularity see e.g. \cite{DaskalopoulosKenig2007,DiBenedetto1993,DiBenedetto1986,Urbano2008,Vazquez2007}.   
	
	\section{Decreasing $L^q$ norm for smooth solutions}
	
	In this section we consider $f$ satisfying the following hypothesis.\\
	
	\noindent {\bf(f1)} $ \, \displaystyle   \frac{\partial f  }{\partial u }  \in L^{\infty}\left(\R^n \times [0,T] \times [-M,M]\right), $ for each $M >0$.\\			
	
	Let $0 < \xi \leq 1 $ and $\zeta_R $ a cut-off function given by $\zeta_R=0$ if $|x|>R$ and \linebreak $\zeta_R=e^{-\xi \sqrt{1+|x|^2}}-e^{-\xi \sqrt{1+R^2}}$ if $|x|\leq R$. Considering $q \geq p_0$ and $ \delta > 0$, we define  $\Phi_\delta (v) := L^q_{\delta}(v) $, where $ L_{\mbox{}_{\scriptstyle \!\:\!\delta}} \!\:\!\in C^{2}(\mathbb{R}) $ by 
	\begin{equation}\label{212}
	L_{\mbox{}_{\scriptstyle \!\:\!\delta}}(\mbox{u})
	:= \int_{0}^{\mbox{\scriptsize u}} \!
	S(\mbox{\mbox{v}}/\delta) \;d\mbox{v},
	\qquad
	\mbox{u} \in \mathbb{R},
	\end{equation}
	where $S(0)=0$, $S(u)=-1$ if $u\leq-1$, $S(u)=1$ if $u\geq1$ and $S$ is smooth and non-decreasing for $-1\leq u\leq 1$.
	
	These cut-off functions will be useful in this Section and in Section 3.

	\begin{teo}\label{307} Let $q \geq p_0$ and $T>0$. If $\, u(x,t) $ is a smooth and bounded solution in $\R^n \times [0,T]$ of \eqref{limitacaolq}
		and $f$ satisfies \emph{{\bf (f1)}}, then \vspace{-0,2cm}
		$$\|u(\cdot,t)\|_{L^q(\R^n)} \leq \|u(\cdot,t_0)\|_{L^q(\R^n)} , \, \, \forall \,  0 \leq t_0 \leq t \leq T ,  \text{ for each } p_0 \leq q \leq \infty. $$
	\end{teo} 
	
	\begin{proof} Let $p_0 \leq q \leq \infty $, $M(T):= \sup\{\|u(\cdot,\tau)\|_{L^{\infty}(\R^n)} : 0 < \tau < T  \} $, $\delta >0$ and $\displaystyle K_f(T):= \sup \left\{\left|\frac{\fr}{\fr v}[f(x,t,v)]\right| : |v|  \leq M(T)\right\}$. Multiplying the first equation of \eqref{limitacaolq} by  $\Phi _{\delta} '(u)\zeta_R(x)$ and integrating in $\R^n \times [t_0,T]$ for each  $0 < t_0 < t \leq T,$ we obtain 
		\begin{align}
		\int _{t_0}^t \int _{|x|<R}\!\!\!\!\!\!\Phi _{\delta} '(u)\zeta_R(x)u_{\tau}\dx \,  d\tau \, + &  \int _{t_0}^t  \int _{|x|<R}\!\!\!\!\!\!\Phi _{\delta} '(u)\zeta_R(x)\Div(f)  \dx \,  d\tau  \nonumber\\ 
		&= \int _{t_0}^t \int _{|x|<R}\!\!\!\!\!\!\Phi _{\delta} '(u)\zeta_R(x)\Div(|u|^{\alpha}\grad u) \dx \,  d\tau.
		\label{eq1}
		\end{align}
		
		\noindent Now, we estimate the second term of the left side.
		\begin{align}\label{4101} 
		\!\!\! \int _{t_0}^t \!\!  \int _{|x|<R} \!\!\!\!\!\!\!\! \Phi _{\delta} '(u)\zeta_R(x)\Div(f)  \dx \,  d\tau 
		\leq  \int _{t_0}^t \!\! \int _{B_R} \!\!\!\!\!\! \Phi _{\delta} '(u) \sum_{j=1}^n  \frac{\fr f_j}{\fr v}(x,t,u)u_{x_j}   \zeta_R(x)  \dx \,  d\tau  &
		\end{align} 
		
		\noindent as $\displaystyle - q \int _{t_0}^t \int _{B_R} \sum_{j=1}^n \frac{\fr f_j}{\fr x_j}(x,t,u) L^{q-1}_{\delta}(u) L'_{\delta}(u) \zeta_R(x)  \dx \,  d\tau \, \leq \, 0, $ by \eqref{i003}.  \\
		
		\noindent Let $\delta \rightarrow 0 $,  by \eqref{4101} and the Divergence Theorem,  we obtain \\\vspace{-0.3cm}
		\begin{align*}
		q \int _{t_0}^t  \int _{|x|<R}\!\!\!\!\!\!|u|^{q-2}u&\zeta_R(x)\Div(f)  \dx \,  d\tau \\
		\leq &  \, - q \int _{t_0}^t \int _{B_R} \sum_{j=1}^n \frac{\fr f_j}{\fr u}(x,t,u)  |u(x,t)|^{q-2}  \, u \, u_{x_j} \, \zeta_R(x)  \dx \,  d\tau  \\
		= & \, - q \int _{t_0}^t \int _{B_R} \Div \left(\int_0^u  \frac{\fr f_j}{\fr v}(x,t,v) \, q \, |v(x,t)|^{q-2}  \, v \,   d \,v \right) \, \zeta_R(x)  \dx \,  d\tau \\
		= & \, \int _{t_0}^t \int _{B_R} \langle \int_0^u \frac{\fr f}{\fr v}(x,t,v) \, q \, |v(x,t)|^{q-2}  \, v  \, dv  \, , \, \grad \zeta_R(x) \rangle \dx \,  d\tau \\
		\leq & \:  K_f(T) \int _{t_0}^t \int _{B_R} |u(x,t)|^{q} \, | \grad \zeta_R(x)| \dx \,  d\tau ,
		\end{align*}
		as
		\begin{equation*}
		\begin{aligned}
		\displaystyle \left|  \int_0^u \frac{\fr f}{\fr v}(x,t,v) \, q \, |v(x,t)|^{q-2}  \, v  \, dv \right| & \leq K_f(T)\int_0^u q \, |v(x,t)|^{q-2}  \, v  \, dv \\
		& = K_f(T)|u(x,t)|^q.\\ 
		\end{aligned}
		\end{equation*}

		\color{black}
		Let be $\displaystyle  {G}_{\delta}(u) := \int_0^{u }|w|^{\alpha}\Phi' _{\delta}(w) \, dw$. As $\displaystyle   \Phi'' _{\delta}(u) \zeta_R (x)|u|^{\alpha}  \langle \grad u  ,  \grad u  
		\rangle \geq 0$, applying the Divergence Theorem in the right side of (\ref{eq1}) we have that
		\begin{align*}
		\int _{t_0}^t \int _{|x|<R}\!\!\!\!\!\!\!\!\Phi _{\delta} '(u)\zeta_R(x)\Div(|u|^{\alpha}\grad u) \dx \,  d\tau
		\leq & - \int_{t_0}^t \int_{|x| \leq R} \!\!\!\!\!\! \Phi' _{\delta}(u) |u|^{\alpha}  \langle \grad \zeta _R (x)  ,  \grad u  \rangle \dx \, \dtau\\
		\leq & - \int_{t_0}^t \int_{|x| \leq R} \langle \grad \zeta _R (x)  ,  \grad {G}_{\delta}(u)  \rangle \dx \, \dtau\\
		\leq &  \int_{t_0}^t \int_{|x| \leq R} \Lap \zeta _R (x)   {G}_{\delta}(u) \dx \, \dtau \\
		& \,- \frac{1}{R} \int_{t_0}^t \int_{|x| = R} \!\!\!\!\!\!\!\!\! {G}_{\delta}(u) \langle \grad \zeta _R (x)  , x   \rangle \, d\sigma (x) \, \dtau .
		\end{align*}
		\noindent Note that\\\vspace{-0.2cm}
		
		$ \displaystyle {G}_{\delta}(u) \leq \int_0^{u(x,t)} M^{\alpha }(T)  \Phi'_{\delta}(w) \, dw \leq M^{\alpha }(T)\Phi_{\delta}(u) $.\\\vspace{-0.2cm}
		
		\noindent This estimate and Cauchy-Schwarz inequality give us \vspace{-0.3cm}
		\begin{align*}
		\int _{t_0}^t \int _{|x|<R}\!\!\!\!\!\!\!\!\!\!\!\Phi _{\delta} '(u)\zeta_R(x)\Div(|u|^{\alpha}\grad u) \dx \,  d\tau &\leq   \, M^{\alpha}(T)  \int_{t_0}^t \int_{|x| \leq R} \!\!\!\!\!\!\!\!\!|\Lap \zeta _R (x) |\Phi_{\delta}(u(x,\tau)) \dx \, \dtau  \\
		& + M^{\alpha}(T)  \int_{t_0}^t \int_{|x| = R} \!\!\!\!\!\!\!\!\! \Phi _{\delta}(u)   | \grad \zeta _R (x) | \, d\sigma (x) \, \dtau  .
		\end{align*}
		
		Using the previous estimates, applying Fubini's Theorem in the first term on the left hand side, and letting $\delta \rightarrow 0$, we obtain
		\begin{align*}
		0 \leq \int _{|x|<R} \!\!\!\!\!\! \zeta_R(x) |u(x,t)|^q  \dx \, \leq &  \int _{|x|<R} \!\!\!\!\!\!  \zeta_R(x) |u(x,t_0)|^q \dx \,  \\
		&+ K_f(T) \int_{t_0}^t \int_{B_R} | \grad \zeta _R (x)| \, |u(x,\tau)|^q \dx \, \dtau \, \\
		&+ M^{\alpha}(T)  \int_{t_0}^t \int_{B_R}\Lap \zeta _R (x) \, |u(x,\tau)|^q \dx \, \dtau \, \\
		&+ M^{\alpha}(T) \int_{t_0}^t \int_{|x| = R} \!\!\!\!\!\!\!\! |u(x,\tau)|^q \, |  \grad \zeta _R (x) | \, d\sigma (x) \, \dtau  .
		\end{align*} 
		The triangular inequality and estimates for $|\grad \zeta _R (x)|$ and $|\Lap \zeta _R (x)|$, give us 
		\begin{align*}
		\int _{|x|<R}  |u(x,t)|^q \zeta _R (x) \dx \, \leq &  \int _{|x|<R}   |u(x,t_0)|^q \zeta _R (x) \dx \, \\
		& +  \xi K_f(T) \int_{t_0}^t \int_{B_R}  |u(x,\tau)|^q \, e^{-\xi\sqrt{1+|x|^2}} \dx \, \dtau \, \\
		& +  n\xi M^{\alpha}(T) \int_{t_0}^t \int_{B_R} \!\!\!\!\!\! |u(x,\tau)|^q e^{-\xi\sqrt{1+|x|^2}} \dx \, \dtau \, \\
		& +  \xi M^{\alpha}(T) \!\!\! \int_{t_0}^t\! \int_{|x| = R} \!\!\!\!\!\! |u(x,\tau)|^q \, e^{-\xi\sqrt{1+R^2}} \, d\sigma (x) \, \dtau ,
		\end{align*}
		Letting $R \rightarrow \infty$ and applying the Gronwall Lemma, we obtain
		\begin{equation*}
		\int _{\R ^n}  |u(x,t)|^q e^{-\xi\sqrt{1+|x|^2}} \dx \, \leq   \int _{\R ^n}  |u(x,t_0)|^q e^{-\xi\sqrt{1+|x|^2}} \dx \, \exp (S(\xi,T,t)),
		\end{equation*}
		where $\displaystyle S(\xi,T,t) = \left( n\xi M^{\alpha}(T)  + \xi K_f(T)  \right)t$, and
		$ S(\xi,T,t) \rightarrow 0 ,$ if $\xi \rightarrow 0$. Then, letting $\xi \rightarrow 0$ and $t_0 \rightarrow 0$ (in this order), we obtain
		\begin{center}
			$\displaystyle \|u(\cdot,t)\|_{L^q(\R^n)}  \, \leq   \|u_0\|_{L^q(\R^n)}  < \infty .$\vspace{-0,75cm}
		\end{center}
	\end{proof}

	\section{Decay estimates for $L^q$ and $L^{\infty}$ norms }\label{decaimento}
	
	In this section we will obtain one important energy inequality. This inequality will be fundamental to obtaining the decay rate for smooth solutions in the last section.

	\begin{teo}\label{9990} Let $q \geq p_0$ and $T>0$. If $\, u(x,t) $ is a smooth and bounded solution in $\R^n \times (0,T]$ of \eqref{limitacaolq}
		and $f$ satisfies \emph{{\bf (f1)}}, then \vspace{-0,2cm}
		\begin{equation*}
		\begin{aligned}
		(t-t_0)^{\gamma}\|u(\cdot,t) \|_{L^q(\R ^n)}^q + q(q-1)\int_{t_0}^t (\tau-t_0) ^{\gamma} \int_{\R ^n} |u(x,\tau)|^{q-2+\alpha} |\grad u|^2 \dx \dtau \quad \,\, \mbox{}  &  \\
		\leq  \gamma \int _{t_0} ^t (\tau-t_0)^{\gamma -1}\|u(\cdot,\tau)\|^q_{L^q(\R ^n)} \, d \tau, &
		\end{aligned}
		\end{equation*}
		where $\gamma > 1 .$
	\end{teo}
	
	\begin{proof} Let $\gamma > 1$,   $T>0$, $M(T):= \sup\{\|u(\cdot,\tau)\|_{L^{\infty}(\R^n)} \,:  \, 0 < \tau < T  \} ,$ $p_0 \leq q < \infty$, $\displaystyle K_f(T):= \sup \left\{\left|\frac{\fr f}{\fr  u} (v) \right| : |v|  \leq M(T)\right\}$ and $\delta >0.$ 
		Multiplying the first equation of the problem \eqref{i001} by $(\tau-t_0)^{\gamma}\Phi _{\delta} '(u)\zeta_R(x)$, integrating on $\R^n \times [t_0,t]$, where $0 < t_0 < t \leq T,$ and applying Fubini's Theorem and integrating by parts the first term on the left side we obtain 
		\begin{align} 
		\displaystyle \!\!  \int _{B_R} \!\!\!\!\!\! \zeta_R(x) (t -t_0)^{\gamma}\Phi _{\delta} (u(x,t))  \dx \, & - \gamma \!\! \int _{t_0} ^t \! \int _{B_R} \!\!\!\!\!\! \zeta_R(x) (\tau-t_0)^{\gamma -1}\Phi _{\delta} (u(x,\tau)) \dx \, d \tau  \nonumber\\
		& +  \int _{t_0}^t \int _{|x|<R}\!\!\!\!\!\!\!\!(\tau -t_0)^{\gamma}\Phi _{\delta} '(u)\zeta_R(x)\Div(f)  \dx \,  d\tau \notag \\
		=  \int _{t_0}^t &\! \int _{|x|<R}\!\!\!\!\!\!\!(\tau -t_0)^{\gamma}\Phi _{\delta} '(u)\zeta_R(x)\Div(|u|^{\alpha}\grad u)  \dx \,  d\tau .
		\label{eq1-teo2}
		\end{align}
		Note that, by \eqref{i003}, we obtain
		\begin{align} \label{e2}
		-\int _{t_0}^t \int _{|x|<R}\!\!\!\!\!\!  (\tau -&t_0)^{\gamma}\Phi _{\delta} '(u)\zeta_R(x)\Div(f)  \dx \,  d\tau \notag \\
		&\leq   - \displaystyle \int _{t_0}^t \int _{|x|<R}\!\!\!\!\!\! (\tau -t_0)^{\gamma}\Phi _{\delta} '(u) \sum_{j=1}^n \frac{\fr f_j}{\fr v}(x,t,u)u_{ x_j}  \zeta_R(x)  \dx \,  d\tau ,
		\end{align}
		\noindent as $\displaystyle - \, q \int _{t_0}^t \int _{B_R}(\tau -t_0)^{\gamma}  \sum_{j=1}^n \frac{\fr f_j}{\fr x_j}(x,t,u) L^{q-1}_{\delta}(u) L'_{\delta}(u) \zeta_R(x)  \dx \,  d\tau \, \leq \, 0. $ 
		
		\noindent Also, writing $\displaystyle  {G}_{\delta}(u) := \int_0^{u }|w|^{\alpha} \Phi' _{\delta}(w) \, dw   $ and applying the Divergence Theorem, we have that the term on the right hand side of (\ref{eq1-teo2}) can be written as
		\begin{flushleft}
			$\displaystyle \int _{t_0}^t \int _{|x|<R}\!\!\!\!\!\!(\tau -t_0)^{\gamma}\Phi _{\delta} '(u)\zeta_R(x)\Div(|u|^{\alpha}\grad u)  \dx \,  d\tau$ \vspace{-0,2cm}
		\end{flushleft}
		\begin{flushright}
			$\displaystyle =- \int_{t_0}^t \int_{|x| \leq R}\!\!\!(\tau - t_0)^{\gamma} \langle \grad \zeta _R (x) ,  \grad {G}_{\delta}(u)  \rangle \dx \, \dtau $ \\
			$\displaystyle-\!q(q-1)\!\! \int_{t_0}^t \! \int_{|x| \leq R}\!\!\!\!\!\!\!\!(\tau - t_0)^{\gamma} L _{\delta}(u)^{q-2}(L' _{\delta}(u))^2 \zeta_R (x)|u|^{\alpha} |\grad u|^2    \dx  \dtau .$
		\end{flushright}

		\noindent Applying one more time the Divergence Theorem and using Cauchy-Schwarz inequality, there holds that
		\begin{align}\label{e3}
		\!\!\!\! \int _{t_0}^t \int _{|x|<R}\!\!\!\!\!\! (\tau -t_0)^{\gamma}\Phi _{\delta} '(u)\zeta_R(x)\Div(|u|^{\alpha}\grad u)  \dx \,  d\tau \hspace{4.2cm} \mbox{}   &  \notag \\	
		\leq  \, M^{\alpha}(T)  \int_{t_0}^t \int_{|x| \leq R}\!\!\!\!\!\!(\tau - t_0)^{\gamma}\Phi _{\delta}(u)  \Lap \zeta _R (x) \dx \, \dtau & \notag \\
		+ \,  M^{\alpha}(T) \int_{t_0}^t \int_{|x| = R}\!\!\!\!\!\!\!\!(\tau - t_0)^{\gamma}  \Phi _{\delta}(u) | \grad \zeta _R (x) | \, d\sigma (x) \, \dtau  \notag \\
		-\, q(q-1)  \int_{t_0}^t \! \int_{|x| \leq R}\!\!\!\!\!\!\!\!(\tau - t_0)^{\gamma} L _{\delta}(u)^{q-2}(L' _{\delta}(u))^2 \zeta_R (x)|u|^{\alpha} | \grad u |^2 \dx \dtau &,
		\end{align}
		as $G_{\delta}(u)\leq M^{\alpha}(T)\Phi_{\delta}(u)$. Substituting (\ref{e2}) and (\ref{e3}) in (\ref{eq1-teo2}),
		\begin{align} 
		\displaystyle  \int _{B_R} \!\!\!\!\!\! \zeta_R(x) (t -t_0)^{\gamma}\Phi _{\delta} (u(x,t))  \dx \, \leq \gamma \!\! \int _{t_0} ^t \! \int _{B_R} \!\!\!\!\!\! \zeta_R(x) (\tau-t_0)^{\gamma -1}\Phi _{\delta} (u(x,\tau)) \dx \, d \tau \quad \, \, \,  \mbox{} & \nonumber\\
		- \int _{t_0}^t \int _{|x|<R}\!\!\!\!\!\!\!\!(\tau -t_0)^{\gamma}\Phi _{\delta} '(u) \sum_{j=1}^n \frac{\fr f_j}{\fr v}(x,t,u)u_{ x_j}  \zeta_R(x)  \dx \,  d\tau & \nonumber\\
		+ M^{\alpha}(T) \int_{t_0}^t \int_{|x| \leq R} \!\!\!\!\!\!\!\!(\tau - t_0)^{\gamma}\Phi _{\delta}(u)  \Lap \zeta _R (x) \dx \, \dtau & \nonumber\\
		+ M^{\alpha}(T) \!\!\int_{t_0}^t\! \int_{|x| = R}\!\!\!\!\!\!\!\!\!\!(\tau - t_0)^{\gamma}  \Phi _{\delta}(u) | \grad \zeta _R (x) | \, d\sigma (x) \, \dtau & \nonumber\\ 
		- q(q-1)\!\! \int_{t_0}^t \! \int_{|x| \leq R}\!\!\!\!\!\!\!\!\!\!(\tau - t_0)^{\gamma} L _{\delta}(u)^{q-2}(L' _{\delta}(u))^2 \zeta_R (x)|u|^{\alpha} | \grad u |^2 \dx  \dtau & .\nonumber
		\end{align}
		
		\noindent As $\| u(\cdot ,t ) \|_{L^q(\R ^n)} < \infty$, following the same steps of Theorem \ref{307}, letting $\delta\rightarrow 0$, $R\rightarrow\infty$ and $\xi \rightarrow 0 $, we obtain
		\begin{align}\label{304}
		\hspace{-0.3cm}(t-t_0)^{\gamma}\|u(\cdot,t) \|_{L^q(\R ^n)}^q + q(q-1)\!\! \int_{t_0}^t \!\! (\tau - t_0) ^{\gamma} \!\!  \int_{\R ^n}\!\!\! |u(x,\tau)|^{q+\alpha-2} |\grad u|^2 \dx \dtau \quad   \mbox{} & \notag  \\ 
		\leq  \gamma \int _{t_0} ^t (\tau - t_0)^{\gamma -1}\|u(\cdot,\tau)\|^q_{L^q(\R ^n)} \, d \tau .& 
		\end{align}
	\end{proof}
	
	The next Theorem shows that $\|u(\cdot ,t)\|_{L^{\infty}(\R^n)}$ decreases in $t$.
	\begin{teo}\label{qwer}Let $\displaystyle q \geq {2p_0}$ and $T>0$. If $\, u(x,t) $ is a smooth and bounded solution in $\R^n \times [0,T]$ of \eqref{limitacaolq}
		and $f$ satisfies \emph{{\bf (f1)}}, then \vspace{-0,2cm}
		$$\|u(\cdot,t)\|_{L^q(\R^n)} \leq K_q(n,\alpha)\|u_0\|_{L^{q/2}(\R^n)}^{\delta}(t-t_0)^{-\kappa}, \, \forall \:  t \in (t_0,T], \, \forall \:  2p_0 \leq q \leq \infty , $$
		where $\displaystyle \delta = \frac{2q + n \alpha}{2q+ 2n \alpha} $ and $\displaystyle \kappa =\frac{n }{2q+ 2n \alpha} \cdot $ 
	\end{teo}
	
	\begin{proof}

		Let  $u(x,t)$ be a smooth solution of \eqref{i001}. Defining $w(x,t) := | \, u(x,t)|^{\frac{q + \alpha}{2}}$, we have $w(\cdot ,t ) \in L^{\beta}(\R^n) \cap L^{\infty}(\R^n) $, 	where $\displaystyle \beta = \frac{2q}{q + \alpha}\,$. By the inequality \eqref{304}, it follows that
		\begin{flushleft}
			$\displaystyle (t-t_0)^{\gamma} \|w(\cdot,t)\|^{\beta}_{L^{\beta}(\R^n)} + \frac{ 4q(q-1)}{(q+\alpha)^2} \! \int _{t_0}^t \!\!\! (\tau - t_0)^{\gamma} \| \grad w(\cdot,\tau) \|^{2}_{L^2(\R^n)} \dtau$
		\end{flushleft}
		\begin{flushright}
			$\displaystyle \leq  \gamma \!\! \int _{t_0}^t \!\!\! (\tau-t_0)^{\gamma -1} \|w(\cdot,\tau)\|^{\beta}_{L^{\beta}(\R^n)} \dtau ,$
		\end{flushright}
		where $\gamma > 1 $ to be chosen.
		By the Nirenberg-Gagliardo-Sobolev's Interpolation Inequality, $\exists \, C >0 $ (constant) such that  
		$$\|w(\cdot,t)\|_{L^{\beta}(\R^n)} \leq C \|w(\cdot,t)\|^{1-\theta}_{L^{\beta /2}(\R^n)}\cdot \|\grad w(\cdot,t)\|^{\theta}_{L^2(\R^n)}, $$
		where  $\displaystyle  \frac{1}{\beta} = \theta \left( \frac{1}{2} - \frac{1}{n}  \right) + (1-\theta) \frac{2}{\beta}    \, \text{. So we have } \,  \theta = \frac{n(q + \alpha) }{nq + 2q + 2n\alpha} \,$ and
		
		\begin{flushleft}
			$\displaystyle (t-t_0)^{\gamma} \|w(\cdot,t)\|^{\beta}_{L^{\beta}(\R^n)} +  \frac{ 4q(q-1)}{(q+\alpha)^2}  \int _{t_0}^t \!\!\! (\tau-t_0)^{\gamma} \| \grad w(\cdot,\tau) \|_{L^2(\R^n)}^2 \dtau \leq$ \\
		\end{flushleft}
		\begin{flushright}
			$\displaystyle  \leq  \gamma C^{\beta} \int _{t_0}^t \!\!\! (\tau - t_0)^{\gamma -1} \|w(\cdot,\tau)\|^{(1-\theta)\beta}_{L^{\beta /2}(\R^n)}\cdot \|\grad w(\cdot,t)\|^{\theta \beta }_{L^2(\R^n)}  \dtau $\\
			$\displaystyle  \leq  \gamma C^{\beta} \|u(\cdot,t_0)\|^{q(1-\theta)}_{L^{q /2}(\R^n)} \int _{t_0}^t \!\!\! (\tau - t_0)^{\gamma -1} \|\grad w(\cdot,t)\|^{\theta \beta }_{L^2(\R^n)}  \dtau , \quad \! \!  $ 
		\end{flushright}
		as $\|w(\cdot,t)\|^{\beta /2}_{L^{\beta /2}(\R^n)} = \|u(\cdot,t)\|^{q/2}_{L^{q /2}(\R^n)} \leq \|u(\cdot,t_0)\|^{q/2}_{L^{q /2}(\R^n)}  $ by Theorem \ref{307}.\\
		
		Applying H\"older's inequality and Young's inequality (in this order) both with $\displaystyle p = \frac{2}{ \theta \beta}$ and $\displaystyle q = \frac{2}{ 2 - \theta \beta} \, $, we obtain\\
		
		\noindent $\displaystyle (t-t_0)^{\gamma} \|w(\cdot,t)\|^{\beta}_{L^{\beta}(\R^n)} + \frac{ 4q(q-1)}{(q+\alpha)^2}  \int _{t_0}^t \!\!\! (\tau- t_0)^{\gamma} \| \grad w(\cdot,\tau) \|_{L^2(\R^n)}^2 \dtau $
		\begin{flushleft}
			$\displaystyle  \leq  \gamma C^{\beta} \|u(\cdot,t_0)\|^{q(1-\theta)}_{L^{q/2}(\R^n)}(t-t_0)^{\frac{2 - \theta \beta}{2}} \left(\int _{t_0}^t \!\!\! (\tau - t_0)^{(\gamma - 1)\frac{2}{\theta \beta} } \|\grad w(\cdot,t)\|^{2}_{L^2(\R^n)}  \dtau  \right)^{\!\!\! \frac{\theta \beta}{2}}$ \\
		\end{flushleft}
		$\displaystyle  \leq   \! \left( \gamma C^{\beta} \|u(\cdot,t_0)\|^{q(1-\theta)}_{L^{q/2}(\R^n)}  \right)^{\! \frac{2}{2- \theta \beta}} \!\frac{2-\theta \beta}{2}\!\! \left(\!  \frac{\theta \beta(q+\alpha)^2}{4q(q-1)}  \right)^{\!\! \frac{\theta \beta}{2}\frac{2}{2-\theta \beta}}(t-t_0)\: + $
		\begin{flushright} 
			$ + \displaystyle  \frac{ 2q(q-1)}{(q+\alpha)^2} \!\! \int _{t_0}^t \!\!\! (\tau- t_0)^{\gamma} \| \grad w(\cdot,\tau) \|_{L^2(\R^n)}^2 \dtau ,$ 
		\end{flushright}
		
		\noindent where in the last inequality we choose  $\gamma$ so that $(\gamma -1)  \frac{2}{\theta\beta} =\gamma$, that is, $\gamma = \frac{2}{2 - \theta \beta}. $ So $\displaystyle \int_0^t \tau^{\gamma -1}\dtau < \infty $, as $\gamma -1 > -1$. Then, 
		
		\noindent $\displaystyle (t-t_0)^{\gamma} \|w(\cdot,t)\|^{\beta}_{L^{\beta}(\R^n)} + \frac{2q(q-1)}{(q+\alpha)^2}   \int _{t_0}^t \!\! (\tau - t_0)^{\gamma}  \| \grad w(\cdot,\tau) \|_{L^2(\R^n)}^2  \dtau  \leq   $
		\begin{flushright}
			$\displaystyle  \leq   \! \left( \gamma C^{\beta} \|u(\cdot,t_0)\|^{q(1-\theta)}_{L^{q/2}(\R^n)}  \right)^{\! \frac{2}{2- \theta \beta}} \!\frac{2-\theta \beta}{2}\!\! \left(\! \frac{\theta \beta (q+\alpha)^2}{4q(q-1)}  \right)^{\!\! \frac{\theta \beta}{2-\theta \beta}}(t-t_0).  $\\\vspace{0.2cm}
		\end{flushright}
		
		\noindent Writing the previous inequality in terms of $u$ we obtain, in particular,\\
		
		\noindent $\displaystyle \|u(\cdot,t)\|^{q}_{L^{q}(\R^n)} \leq \! \left( \gamma C^{\beta} \|u(\cdot,t_0)\|^{q(1-\theta)}_{L^{q/2}(\R^n)}  \right)^{\! \frac{2}{2- \theta \beta}} \!\frac{2-\theta \beta}{2}\!\! \left(\! \frac{\theta \beta (q+\alpha)^2}{4q(q-1)}  \right)^{\!\! \frac{\theta \beta}{2-\theta \beta}}\!\!\!\!(t-t_0)^{1-\gamma},  $\\
		
		\noindent that is,
		
		\noindent $\displaystyle \|u(\cdot,t)\|_{\mbox{}_{L^{q}(\R^n)}} \leq \! (C^{\beta} \gamma)^{{\! \frac{1}{q} \frac{2}{2- \theta \beta}}}\!\! \left( \!\! {\frac{2-\theta \beta}{2}}\right)^{\!\! \frac{1}{q}} \!\!\! \left(\! \frac{\theta \beta (q+\alpha)^2}{4q(q-1)}  \right)^{\!\!\frac{1}{q} \frac{\theta \beta}{2-\theta \beta}} \!\!\!\!\!\! \|u(\cdot,t_0)\|^{\frac{2(1-\theta)}{2 - \theta \beta}}_{\mbox{}_{L^{q/2}(\R^n)}}\!(t-t_0)^{\frac{1-\gamma}{q}} \! .$\\
		
		\noindent As $ \displaystyle \frac{2(1 - \theta)}{2 - \theta \beta} = \frac{2q + n \alpha}{2q+ 2n \alpha} $ and $\displaystyle \frac{1- \gamma}{q} =  \frac{-n}{2q + 2n \alpha}, $ we get\\
		
		\noindent $$\displaystyle \|u(\cdot,t)\|_{\mbox{}_{L^{q}(\R^n)}} \leq K_q  \|u(\cdot,t_0)\|^{\frac{2q + n \alpha}{2q+ 2n \alpha}}_{\mbox{}_{L^{q/2}(\R^n)}}\!(t-t_0)^{ \frac{-n}{2q + 2n \alpha}} , $$
		
		\noindent where $\displaystyle  K_q = K_q(n,\alpha) = (C^{\beta} \gamma)^{{ \! \frac{1}{q} \frac{2}{2- \theta \beta}}} \left(  {\frac{2-\theta \beta}{2}}\right)^{ \! \frac{1}{q}} \! \left( \frac{\theta \beta (q+\alpha)^2}{4q(q-1)}  \right)^{\! \frac{1}{q} \frac{\theta \beta}{2-\theta \beta}}. \!\!\! $ 
	\end{proof}
	
	\begin{teo}\label{qq/2}Let $q \geq p_0$ and $T>0$. If $\, u(x,t) $ is a smooth and bounded solution in $\R^n \times [0,T]$ of \eqref{limitacaolq}
		and $f$ satisfies \emph{{\bf (f1)}}, then \vspace{-0,2cm}
		$$\|u(\cdot,t)\|_{L^{\infty}(\R^n)} \leq K_n(\alpha, q)\|u_0\|_{L^{q}(\R^n)}^{\delta}t^{-\kappa}, \, \forall \:  t \in (0,T], \, \forall \:  p_0 \leq q \leq \infty , $$
		where $\displaystyle \delta = \frac{2q }{2q+ n \alpha}, \, \displaystyle \kappa =\frac{n }{2q+ n \alpha}$ and $K_n(\alpha,q)$ is constant. 
	\end{teo}
	
	\begin{proof} Let $u(\cdot,t )$ be a smooth solution of \eqref{i001}. By Theorem \ref{qwer}\\
		
		\noindent $$\displaystyle \|u(\cdot,t)\|_{\mbox{}_{L^{q}(\R^n)}} \leq K_q  \|u(\cdot,t_0)\|^{\frac{2q + n\alpha}{2q + 2n\alpha}}_{\mbox{}_{L^{q/2}(\R^n)}}\!(t-t_0)^{\frac{-n }{2q+ 2n \alpha}}\!\! , $$ 
		
		\noindent where $ \displaystyle K_q(n,\alpha) = C^{\frac{1}{q+\alpha}\frac{nq+2q+2n\alpha}{q+n\alpha}}\left(\frac{q}{q+n\alpha}\right)^{{\! \frac{1}{2q + 2n\alpha}}}\!\! \left(\! \frac{(q+\alpha)^2}{4q(q-1)}  \right)^{\!\!\frac{n}{2q+2n\alpha}} ,$\\
		
		\noindent   as $\displaystyle  \gamma = \frac{2}{2 - \theta \beta  }$, $\displaystyle \beta = \frac{2q}{q+\alpha}$ and $ \displaystyle  \theta \beta = \frac{2qn}{nq + 2q+ 2n\alpha}$.\\ 
		
		\noindent Let $m \in \N$ with $m \geq 1$, we define $t_0^{(m)} = 2^{-m}t$ and $t_j^{(m)} = t_0^{m} + (1-2^{-j})t$ for all  $1 \leq j \leq m$.   Applying this inequality, $m$ times, substituting $q$ with $2^jq$ and taking $t_0 = t^{(j)}_{j-1}$ , for each $1 \leq j \leq m$, we obtain
		
		\begin{flushleft}
			$\displaystyle \|u(\cdot,t^{(m)}_{m})\|_{\mbox{}_{L^{2^mp_0}(\R^n)}}  \leq   K_{m}  \|u(\cdot,t^{(m)}_{m-1})\|^{\frac{2^m2q + n\alpha}{2^m2q + 2n\alpha}}_{\mbox{}_{L^{2^{m-1}q}(\R^n)}}\!(t-t^{(m)}_{m-1})^{\frac{-n }{2^m2q+ 2n \alpha}}$
		\end{flushleft}
		\begin{flushleft}
			$\displaystyle \leq    K_{m} K^{\frac{2^m2q + n\alpha}{2^m2q + 2n\alpha}}_{m-1}  \|u(\cdot,t^{(m)}_{m-2})\|^{ \frac{2^{m-1}2q + n\alpha}{2^{m-1}2q + 2n\alpha}\frac{2^m2q + n\alpha}{2^m2q + 2n\alpha}}_{\mbox{}_{L^{2^{m-2}q}(\R^n)}}(t_m^{(m)}-t^{(m)}_{m-1})^{\frac{-n }{2^m2q+ 2n \alpha}}\cdot$
		\end{flushleft}
		\begin{flushright}
			$\displaystyle \cdot  (t^{(m)}_{m-1}-t^{(m)}_{m-2})^{ \frac{-n }{2^{m-1}2q+ 2n \alpha}\frac{2^m2q+n\alpha}{2^m2q+2n\alpha}} $\\
		\end{flushright}
		$\vdots$  \\
		$\displaystyle \leq    K_{m}K_{m-1}^{\frac{2^m2q + n\alpha}{2^m2q + 2n\alpha}} \cdot \ldots \cdot K_1^{B_{m-1}}  \|u(\cdot,t_0^{(m)})\|^{A_m}_{\mbox{}_{L^{q}(\R^n)}} \displaystyle   (t_m^{(m)}\!-\! t^{(m)}_{m-1})^{\frac{-n }{2^m2q+ 2n \alpha}B_0}\cdot $\linebreak \vspace{-0,5cm}
		\begin{flushright}
			$\cdot(t^{(m)}_{m-1}\!-\! t^{(m)}_{m-2})^{ \frac{-n }{2^{m-1}2q+ 2n \alpha}B_1} \cdot \ldots \cdot  (t^{(m)}_{1}\!-\!t^{(m)}_{0}\! )^{ \frac{-n }{2^{m-1}2q+ 2n \alpha}B_{m-1}} , $
		\end{flushright}\vspace{0.1cm}
		\noindent where, for each $1 \leq j \leq m,$		
		\noindent $\displaystyle K_{j} \leq C^{\frac{1}{2^jq+\alpha}\frac{2^jq(n+2)+2n\alpha}{2^jq+n\alpha}}\!\! \left( \frac{(2^jq+\alpha)^2}{2^j4q(2^jq-1)}  \right)^{\!\!\!\frac{n}{2^j2q+2n\alpha}}\!\!\! , $
		
		\noindent \\ $\displaystyle A_{m}= \prod_{j=1}^{m} \frac{2^j2q+n\alpha}{2^j2q+2n\alpha} = \frac{1}{2^m}\prod_{j=1}^{m} \frac{2^j2q+n\alpha}{2^{j-1}2q+n\alpha} = \frac{1}{2^m}\frac{2^m2q+n\alpha}{2q + n\alpha} \: $, $B_0 =1$ and \\
		
		\noindent \\ $\displaystyle B_{j}= \prod_{k=0}^{j-1} \frac{2^{m-k}2q+n\alpha}{2^{m-k}2q+2n\alpha} = \frac{1}{2^j}\frac{2^{m}2q+n\alpha}{2^{m-j}2q+n\alpha}   $ for $1 \leq j \leq m$.\vspace{0.2cm}
		
		\noindent Since $t^{(m)}_{j} - t^{(m)}_{j-1} = 2^{-j}t$, for all $1 \leq j \leq m$, we can rewrite the previous inequality as 
		$$\displaystyle \|u(\cdot,t^{(m)}_{m})\|_{\mbox{}_{L^{2^mq}(\R^n)}} \leq \prod_{j=1}^{m} \left[  K_{j}^{B_{m-j}} \|u(\cdot,t_0^{(m)})\|^{A_m}_{\mbox{}_{L^{q}(\R^n)}}(2^{-j}t)^{\frac{-n}{2^j2q+2n\alpha}B_{m-j}} \right] .$$
		\noindent Now let us estimate, separately, $\displaystyle \prod_{j=1}^{m}   K_{j}^{B_{m-j}}$ and $\displaystyle \prod_{j=1}^{m} (2^{-j}t)^{\frac{-n}{2^j2q+2n\alpha}B_{m-j}}.$\\
		
		\noindent Note that  $\displaystyle \prod_{j=1}^{m} (2^{-j}t)^{\frac{-n}{2^j2q+2n\alpha}B_{m-j}} = \prod_{j=1}^{m} t^{\frac{-n}{2^j2q+2n\alpha}B_{m-j}} \prod_{j=1}^{m} 2^{-j{\frac{-n}{2^j2q+2n\alpha}B_{m-j}}}.  $
		
		\noindent First, let us observe that,\\
		
		$\displaystyle  \prod_{j=1}^{m} t^{\frac{-n}{2^j2q+2n\alpha}B_{m-j}} = t^{  \sum_{j=1}^{m} \frac{-n}{2^j2q+2n\alpha}B_{m-j}} $, and that (changing $m-j$ with $j$), we obtain
		\begin{equation*}
		\begin{aligned}
		\displaystyle \sum_{j=1}^{m} \frac{-n}{2^j2q+2n\alpha}B_{m-j} & = \sum_{j=0}^{m-1} \frac{-n}{2^{m-j}2q+2n\alpha}B_{j}\\
		& =  \sum_{j=0}^{m-1} \frac{-n}{2^{m-j}2q+2n\alpha} \frac{1}{2^j}\frac{2^m2q + n \alpha}{2^{m-j}2q + n \alpha}\\
		& = - 2n(2^m2q + n \alpha) \sum_{j=0}^{m-1} \frac{1}{2^{m-j+1}2q+2n\alpha} \frac{2^{-j}}{2^{m-j}2q + 2n \alpha}.
		\end{aligned}
		\end{equation*}
		Defining $\displaystyle  \hat{\alpha} = \frac{2^m2q}{2n \alpha}$, we get  
		\begin{align*}
		\displaystyle \sum_{j=1}^{m} \frac{-n}{2^j2q+2n\alpha}B_{m-j} & = - \frac{2n(2^m2q + n \alpha)}{\hat{\alpha}(2n\alpha)^2} \sum_{j=0}^{m-1} \frac{1}{\hat{\alpha}2^{-j+1} + 1 } \frac{\hat{\alpha}2^{-j}}{\hat{\alpha}2^{-j} + 1}\\
		& =  - \frac{2n(2^m2q + n \alpha)}{\hat{\alpha}(2n\alpha)^2} \sum_{j=0}^{m-1} \frac{1}{\hat{\alpha}2^{-j} + 1 } -  \frac{1}{\hat{\alpha}2^{-j+1} + 1}\\
		& =  - \frac{2n(2^m2q + n \alpha)}{\hat{\alpha}(2n\alpha)^2} \left[ \frac{1}{\hat{\alpha}2^{-m+1} + 1} - \frac{1}{2\hat{\alpha} + 1 }\right]\\
		& =  - \frac{2n(2^m2q + n \alpha)}{2^m2q} \left[  \frac{1}{4q +  2n\alpha} - \frac{1}{2^m4q +  2n\alpha }  \right]\\
		& =  - \frac{2n(2q + \frac{n \alpha}{2^m})}{2q} \left[  \frac{1}{4q + 2 n\alpha} - \frac{1}{2^m4q + 2 n\alpha}  \right] .\\
		\end{align*}
		Now, letting $m \rightarrow + \infty$, we obtain\\
		
		$\displaystyle \sum_{j=1}^{\infty } \frac{-n}{2^j2q+2n\alpha}B_{m-j} = \frac{-2n}{4q + 2 n\alpha} = \frac{-n}{2q +  n\alpha} \,, $ and\\
		
		$\displaystyle \lim_{m \rightarrow + \infty} A_m = \lim_{m \rightarrow + \infty} \frac{2q + \frac{n\alpha}{2^m}}{2q +n\alpha} = \frac{2q}{2q+ n\alpha} .$\\
		
		\noindent So, we show that, in fact,  $\displaystyle \delta =\frac{2q}{2q+ n\alpha} $ and $\displaystyle \kappa = \frac{n}{2q +  n\alpha} $. It remains to obtain a bound for $K_n(\alpha,q)$ independent of $m$.\\
		
		\noindent Note that\\
		
		$\displaystyle  \prod_{j=1}^{m} (2^{-j})^{\frac{-n}{2^j2q+2n\alpha}B_{m-j}} =  2^{\sum_{j=1}^{m} j \frac{n}{2^j2q+2n\alpha}B_{m-j}} .  $\\
		
		\noindent As $ \displaystyle B_j =\frac{1}{2^j}\frac{2^m2q + n\alpha}{2^{m-j}2q + n\alpha}  \leq \frac{2^m2q + n\alpha}{2^{m}2q} = 1+ \frac{ n\alpha}{2^{m}2q }, \: \forall \, 1 \leq j \leq m ,$ we have\\ 
		
		$\displaystyle  \prod_{j=1}^{m} (2^{-j})^{\frac{-n}{2^j2q+2n\alpha}B_{m-j}} \leq   2^{ n(1+ \frac{ n\alpha}{4q }) \sum_{j=1}^{m} \frac{j}{2^j2q+2n\alpha} } \leq 2^{ \frac{n}{2q}(1+ \frac{ n\alpha}{4q }) \sum_{j=1}^{m} \frac{j}{2^j} } . $\\
		
		\noindent So, letting $m \rightarrow \infty $, we obtain\\
		
		$\displaystyle  \prod_{j=1}^{\infty} (2^{-j})^{\frac{-n}{2^j2q+2n\alpha}B_{m-j}}  \leq 2^{ \frac{n}{2q}(1+ \frac{ n\alpha}{4q }) \sum_{j=1}^{\infty} \frac{j}{2^j} } = 2^{ \frac{n}{q}(1+ \frac{ n\alpha}{4q })} . $\\	
		\noindent Finally, we estimate $\displaystyle \prod_{j=1}^{m}   K_{j}^{B_{m-j}}$. To this end, note that 
		\begin{align*}
		\displaystyle K_{j} & \leq C^{\frac{1}{2^jq+\alpha}\frac{2^jq(n+2)+2n\alpha}{2^jq+n\alpha}}\!\!  \left(\! \frac{(2^jq+\alpha)^2}{2^j4q(2^jq-1)}  \right)^{\!\!\!\frac{n}{2^j2q+2n\alpha}}\!\!\!  \leq C^{\frac{n+2}{2^j2q} +\frac{2n\alpha}{2^{2j}q}}\!\!  \left(\! \frac{(2^jq+\alpha)^2}{2^j4q(2^jq-1)}  \right)^{\!\!\!\frac{n}{2^j2q}} . \\
		\end{align*}
		Defining $j \geq j_0$ so that $\displaystyle \frac{ 2n\alpha}{2^{j}q} < 1 $, $\displaystyle \frac{ \alpha}{2^{j}(2q-1)} < 1 $ and $\displaystyle \frac{ \alpha ^2}{2^{2j}2q(2q-1)} < 1 $, we get
		\begin{align*}
		\displaystyle K_{j} & \leq C^{\frac{n+2}{2^j2q} +\frac{2n\alpha}{2^{2j}q}} \left( \frac{(2^jq+\alpha)^2}{2^j4q(2^jq-1)}  \right)^{\frac{n}{2^j2q}}  \\
		& \leq C^{\frac{n+2}{2^j2q} +\frac{2n\alpha}{2^{2j}q}}  \left( \frac{q}{2(2q-1)} + 2 \right)^{\frac{n}{2^j2q}} .       
		\end{align*}
		So
		\begin{equation*}
		\begin{aligned}
		\displaystyle \prod_{j=1}^{m}   K_{j}^{B_{m-j}} & \leq  \prod_{j=1}^{m}  \left[ C^{\frac{n+2}{2^j2q} +\frac{2n\alpha}{2^{2j}q}} \left( \frac{q}{2(2q-1)} + 2 \right)^{\frac{n}{2^j2q}} \right] ^{\frac{4q + n \alpha}{4q}}\\
		& =\left(\prod_{j=1}^{j_0 -1 }   K_{j}^{\frac{4q + n\alpha}{4q}}\right)  C^{\frac{4q + n \alpha}{4q} \frac{n+2}{2q}(1- 2^{-m}) + \frac{2n \alpha }{q } (\frac{1}{3} - 4^{-m})}\cdot\\
		& \hspace{5,1cm} \cdot \left( \frac{q}{2(2q-1)} + 2 \right)^{\frac{4q + n \alpha}{4q}\frac{n}{2q}(1- 2^{-m})  }  .
		\end{aligned}
		\end{equation*}
		\noindent Now, letting $m \rightarrow + \infty$, we obtain\\
		
		\noindent $\displaystyle \prod_{j=1}^{\infty}   K_{j}^{B_{m-j}} \!\!  \leq \! \left(\prod_{j=1}^{j_0 -1 }   K_{j}^{\frac{4q + n\alpha}{4q}}\right) \!  C^{\frac{4q + n \alpha}{4q} \frac{n+2}{2q} + \frac{2n \alpha }{3q }} \! \left( \frac{q}{2(2q-1)} + 2 \right)^{\!\! \frac{4q + n \alpha}{4q}\frac{n}{2q}  } < \infty.$
	\end{proof}

	\section{Optimal rate for signed solutions}

	\quad In this section we will obtain an optimal rate decay for weak solutions of the problem  
	\begin{align}\label{S}
	u_t + \Div  f(x,t,u) &= \Div ( |u(x,t)|^{\alpha} \grad u ) \quad x \in \R^{n}, \, t>0, \notag  \\
	u(\cdot,0) &=u_0  \in L^{p_0}(\R^n)\cap L^{\infty}(\R^n) .  
	\end{align}
	where $u_0$ is any function in $L^{p_0}(\R)$.
	
	Let us consider the auxiliary problems 
	\begin{align}\label{P}
	u_t + \Div  f(x,t,u)  &= \Div ( |u(x,t)|^{\alpha} \grad u ) \quad x \in \R^{n}, \, t>0, \notag \\
	u(\cdot,0) &=u_0^+ + \epsilon \psi  \in L^{p_0}(\R^n)\cap L^{\infty}(\R^n) .  
	\end{align}
	and 
	\begin{align}\label{N}
	u_t + \Div  f(x,t,u) &= \Div ( |u(x,t)|^{\alpha} \grad u ) \quad x \in \R^{n}, \, t>0, \notag \\
	u(\cdot,0) &=-u_0^- - \epsilon \psi \in L^{p_0}(\R^n)\cap L^{\infty}(\R^n) . 
	\end{align}
	where $\epsilon >0 $ and $ 0 < \psi \in L^{p_0}(\R^n)\cap L^{\infty}(\R^n) .  $\\
	
	To prove the main result of this article, we need $ f $ to satisfy the hypothesis below\\
	
	\noindent {\bf(f2)} $ \, \displaystyle |\, f(x,t,\mbox{u}) -
	\!\;\!f(x,t,\mbox{v}) \,|
	\;\leq\,
	C_{\!f}(\mbox{\small $M$}\!\;\!,\mbox{\small $T$}) \:
	|\, \mbox{u} - \mbox{v} \,| \, ,
	\quad
	\forall \;
	x \in \mathbb{R}^{n}\!\:\!, \;
	0 \leq t \leq \mbox{\small $T$}\!\;\!. $ \\
	
	\begin{teo} Let $q \geq p_0$ and $T>0$. If $\, u(x,t) $ is a weak and bounded solution in $\R^n \times [0,T]$ of \eqref{S}
		and $f$ satisfies \emph{{\bf (f1 - f2)}}, then \vspace{-0,2cm}
		$$\|u(\cdot,t)\|_{L^{\infty}(\R^n)} \leq K_n(\alpha, q)\|u_0\|_{L^{q}(\R^n)}^{\delta}t^{-\kappa}, \, \forall \:  t \in (0,T], \, \forall \: p_0 \leq q \leq \infty , $$
		where $\displaystyle \delta = \frac{2q }{2q+ n \alpha}, \, \displaystyle \kappa =\frac{n }{2q+ n \alpha}$ and $K_n(\alpha,q)$ is constant. 
	\end{teo}
	
	\begin{proof} Let $u,v,w$ be, respectively, solutions of \eqref{S}, \eqref{P} and \eqref{N}, by comparison (see \cite{Fabris2013}), we have $w(x,t) \leq u(x,t) \leq v(x,t), \,  w(x,t) \leq 0$ and $ 0 \leq  v(x,t), \, \forall  x \in \R^n$ and $\forall \, t>0$. Now, by Theorem \ref{qq/2}, the estimates obtained for smooth solutions, that do not change sign, of problems with initial data that do not change sign, it follows that    
		$$\displaystyle \| w(\cdot,t) \|_{L^{\infty}(\R^n)} \leq K_n(\alpha,n)\|-u_0^- - \epsilon \psi \|^{\delta}_{L^q(\R^n)}t^{-\kappa}, \, \forall t \in (0 , T], \, \forall \, p_0 \leq q \leq \infty, $$ where $ \displaystyle \delta = \frac{2q}{2q + n\alpha}, \displaystyle \kappa = \frac{n}{2q + n\alpha}  $  and $K_n(\alpha,q)$ is constant. In particular, we have
		\begin{align*}
		- K_n(\alpha,n)\|-u_0^- - \epsilon \psi \|^{\delta}_{L^q(\R^n)}t^{-\kappa} \leq w(x,t) & \leq u(x,t) \\
		& \leq v(x,t) \\
		& \leq K_n(\alpha,n)\|u_0^+ + \epsilon \psi \|^{\delta}_{L^q(\R^n)}t^{-\kappa}  ,
		\end{align*}
		hence
		$$ |u(\cdot, t )| \leq K_n(\alpha,n) \max \left\{ \|-u_0^- - \epsilon \psi \|^{\delta}_{L^q(\R^n)},\|u_0^+ + \epsilon \psi \|^{\delta}_{L^q(\R^n)}  \right\} t^{-\kappa}, \, \forall \epsilon >0 .  $$
		
		As $\displaystyle \max \left\{ \|-u_0^- - \epsilon \psi \|^{\delta}_{L^q(\R^n)},\|u_0^+ + \epsilon \psi \|^{\delta}_{L^q(\R^n)}  \right\} $ decreases, when $\epsilon \rightarrow 0^+$, we obtain 
		$$ \|u(\cdot, t )\|_{L^{\infty}(\R^n)} \leq K_n(\alpha,n) \max \left\{ \|u_0^- \|^{\delta}_{L^q(\R^n)},\|u_0^+  \|^{\delta}_{L^q(\R^n)}  \right\} t^{-\kappa}.  $$\vspace{-0.75cm}
	\end{proof}

\end{document}